\documentclass{amsart}

\usepackage{amssymb,amsmath,amsthm}
\usepackage{graphicx}

\newtheorem{thm}{Theorem}[section]
\newtheorem{coro}[thm]{Corollary}
\newtheorem{lem}[thm]{Lemma}

\newtheorem{prob}{Problem}

\numberwithin{equation}{section}
\numberwithin{enumi}{section}

\newcommand{\C}{\mathrm{Cl}_X}
\newcommand{\cl}[2]{\mathrm{Cl}_{#1}\!\left(#2\right)}
\newcommand{\Ord}{\mathrm{Ord}_X}
\newcommand{\Orb}{\mathrm{Orb}_X}
\newcommand{\m}[1]{\mathbb{#1}}
\newcommand{\N}{\mathbb{N}}
\newcommand{\ho}{ \mathcal{H}(X)}

\begin{document}

\title{Homogeneity degree of fans}

\author[G. Acosta]{Gerardo Acosta}
\address[G. Acosta]{Instituto de Matem\'{a}ticas, Universidad Nacional Aut\'{o}noma de M\'{e}xico, Ciudad Universitaria, D.F. 04510, Mexico}
\email{gacosta@matem.unam.mx}

\author[L. C. Hoehn]{Logan C. Hoehn}
\address[L.C. Hoehn]{Department of Computer Science and Mathematics, Nipissing University, 100 College Drive, Box 5002, North Bay, Ontario, Canada, P1B 8L7}
\email{loganh@nipissingu.ca}

\author[Y. Pacheco Ju\'arez]{Yaziel Pacheco Ju\'arez}
\address[Y. Pacheco]{Department of Computer Science and Mathematics, Nipissing University, 100 College Drive, Box 5002, North Bay, Ontario, Canada, P1B 8L7}
\email{yazielpj@nipissingu.ca}

\begin{abstract}
The homogeneity degree of a topological space $X$ is the number of orbits of the action of the homeomorphism group of $X$ on $X$.  We initiate a study of dendroids of small homogeneity degree, beginning with fans.  We classify all smooth fans of homogeneity degree $3$, and discuss non-smooth fans and prove some results on degree $4$.
\end{abstract}

\thanks{The second named author was supported by NSERC grant RGPIN 435518.  The third named author was partially supported by CONACYT grant 316891.}

\subjclass[2010]{Primary 54F50; Secondary 54F65, 54F15.}
\keywords{Homogeneity degree, Continuum, Fan, Smooth dendroid}

\maketitle

\section{Introduction}

For a topological space $X$ we denote by $\ho$ the group of homeomorphisms of $X$ onto itself.  Given $x \in X$, we denote by $\Orb(x)$ the orbit of $x$ under the action of $\ho$ on $X$; that is, $\Orb(x) = \{h(x): h \in \ho\}$.  The \emph{homogeneity degree} of $X$ is the number of orbits for the action of $\ho$ on $X$.  Alternatively, we say $X$ is \emph{$\frac{1}{n}$-homogeneous} to mean that $X$ has homogeneity degree $n$.

A space is \emph{homogeneous} if its homogeneity degree is $1$.  Homogeneity is a classical and well-studied notion in continuum theory; however, some important classes of spaces (e.g.\ dendroids) include no homogeneous spaces, and thus do not interact with this theory.  In this paper, we aim to demonstrate that an appropriate notion of homogeneity for dendroids, especially for fans, is $\frac{1}{3}$-homogeneous.

An emerging theme in the study of continua (or more generally compact metric spaces) is that spaces of low homogeneity degree tend to be rare and remarkable, and conversely many well-known and interesting compact metric spaces have a low degree of homogeneity.  This can be seen in recent work, e.g.\ in the classification of all homogeneous plane continua (more generally compact spaces in the plane) completed in \cite{HoehnOversteegen}, and the classification of all $\frac{1}{3}$-homogeneous dendrites \cite{13-dendrites}.

In this paper we initiate a study of dendroids of low homogeneity degree, beginning with fans.  We classify all $\frac{1}{3}$-homogeneous smooth fans, and prove that there are no smooth fans with homogeneity degree $4$.  We also consider some properties of non-smooth fans of homogeneity degrees $3$ or $4$.  It is not yet known whether these exist -- see Problem \ref{nonsmooth exist}.

\subsection{Definitions and notation}

For a topological space $X$ and $A \subset X$ the symbol $\C(A)$ denotes the closure of $A$ in $X$.  In this paper, all spaces considered will be metric spaces, and the metric will always be denoted by $d$.  When we refer to the distance or convergence of closed subsets of $X$, it will be understood that we are considering the Hausdorff metric.

A \emph{continuum} is a compact connected metric space.  An \emph{arc} is a space homeomorphic to the interval $[0,1] = I$.  We call any space homeomorphic to the standard middle-third Cantor set a \emph{Cantor set}.  A \emph{dendroid} is an arcwise connected and hereditarily unicoherent continuum.  Given a dendroid $X$ and points $p,q \in X$, we denote by $pq$ the unique arc in $X$ having $p$ and $q$ as its endpoints, and put $(pq) = pq - \{p,q\}, [pq) = pq - \{q\}$ and $(pq] = pq - \{p\}$.

For a dendroid $X$ we say that $p$ has \emph{order} $r$ (in the classical sense), in symbols $\Ord(p) = r$, if $p$ is a common endpoint of exactly $r$ arcs in $X$ which are disjoint from one another beyond $p$.  A point $p \in X$ is an \emph{endpoint} of $X$ if $\Ord(p) = 1$, a \emph{ramification point} of $X$ if $\Ord(p) > 2$, and an \emph{ordinary point} if $\Ord(p) = 2$.  We denote by $E(X)$, $R(X)$, and $O(X)$ the set of endpoints, the set of ramification points, and the set of ordinary points of $X$ respectively.  It is clear that any homeomorphism of $X$ onto itself preserves these three sets.

A \emph{fan} is a dendroid $X$ with exactly one ramification point called the \emph{top} of $X$.  Note that in a fan $X$ the sets $E(X)$, $O(X)$, and $R(X)$ are all nonempty, hence any fan has homogeneity degree greater than or equal to $3$.  If $X$ is a fan with top $t$, then $X = \bigcup_{e \in E(X)} te$ and $te \cap td = \{t\}$ for any two distinct $e,d \in E(X)$.  It is clear that $O(X)$ is dense in $X$.  From \cite[Theorem 7.5, p.311]{Lelek}, we have that $E(X)$ is a $G_\delta$ subset of $X$.

A dendroid $X$ is \emph{smooth} if there is a point $p \in X$ such that for each sequence $\{a_n\}_n$ of points of $X$ which converges to a point $a \in X$, the sequence of arcs $\{pa_n\}_n$ converges to the arc $pa$.  In a smooth fan, the top of the fan can always be used for the point $p$ in the previous definition.

\subsection{Effros' Theorem}

We recall some fundamental notions from descriptive set theory.  A \emph{Polish space} is a separable, completely metrizable space.  A subset of a Polish space $X$ is itself Polish if and only if it is a $G_\delta$ subset of $X$.  If $X$ is a compact metric space, then $\ho$ is a \emph{Polish group}, i.e.\ a separable, completely metrizable topological group.

It is a consequence of a theorem of Effros \cite[Theorem 2.1, p.39]{Effros} that if $X$ is a Polish space and $G$ is a Polish group acting transitively on $X$, then for any $x \in X$ and any neighborhood $U$ of the identity in $G$, the set $Ux = \{g \cdot x: g \in U\}$ is an open neighborhood of $x$ in $X$.  This result has seen extensive use in the study of homogeneous continua since it was first applied in the area by Ungar in \cite{Ungar}.  We will demonstrate in this paper that it is also useful in the study of spaces with homogeneity degree greater than $1$.  We will use the following consequence ($\dagger$) of the theorem.

An \emph{$\varepsilon$-homeomorphism} is an element $h \in \ho$ such that $d \big( x,h(x) \big) < \varepsilon$ for all $x \in X$.

\begin{itemize}
\item[($\dagger$)] Let $X$ be a compact metric space and let $x \in X$.  Suppose $\Orb(x)$ is a $G_\delta$ subset of $X$.  Then for any $\varepsilon > 0$ there exists $\delta > 0$ such that for any $v \in X$ with $d(x,v) < \delta$ there exists an $\varepsilon$-homeomorphism $h: X \to X$ such that $h(x) = v$.
\end{itemize}

\section{Preliminary results}

In \cite[Theorem 2.1, p.302]{Lelek}, it is proved that if $X$ is a dendroid, $E(X)$ does not contain any non-degenerate continua.

When we say that a subset $A$ of $X$ is an \emph{orbit of $X$}, we mean that $A$ is an orbit under the action of $\ho$ on $X$.

\begin{lem}
\label{lemanlc cantor}
Let $X$ be a smooth fan which is not locally connected.  If $E(X)$ is an orbit of $X$, then $O(X) \cap \C \big( E(X) \big) \neq \emptyset$ or $\C \big( E(X) \big)$ is a Cantor set.
\end{lem}

\begin{proof}
Assume that $O(X) \cap \C \big( E(X) \big) = \emptyset$.

Denote by $N(X)$ the set of points where $X$ is not locally connected.  By \cite[Theorem 5.12, p.76]{Nadler92}, $N(X)$ contains a \emph{continuum of convergence}, i.e.\ a non-degenerate continuum $K \subset N(X)$ for which there is a sequence $\{A_n\}_n$ of pairwise disjoint subcontinua of $X$ converging to $K$ such that, for each $n \in \N$, we have $K \cap A_n = \emptyset$.  Since $E(X)$ does not contain any non-degenerate continuum, neither does $E(X) \cup \{t\}$, and so $K \cap O(X) \neq \emptyset$.

Choose $x \in K \cap O(X)$.  We can assume $t \notin A_n$ for all $n$.  Then for each $n$ there is $e_n \in E(X)$ such that $A_n \subset te_n$.  We can assume (by taking a subsequence if needed) that the sequence $\{e_n\}_n$ converges to a point $e \in X$.  Because $X$ is smooth, $te_n \to te$.  Note that $x \in (te)$ and so $e \neq t$.  Also, $e \notin O(X)$ since $O(X) \cap \C \big( E(X) \big) = \emptyset$.  Therefore, $e \in E(X)$.  Since $E(X)$ is an orbit, this means that each point in $E(X)$ is a limit of a sequence of endpoints.

Note that by our initial assumption $\C \big( E(X) \big) \subset E(X) \cup \{t\}$.  Since $E(X)$ does not contain any non-degenerate continuum, neither does $E(X) \cup \{t\}$.  Hence, $\C \big( E(X) \big)$ is totally disconnected and perfect, and consequently it is a Cantor set.
\end{proof}

\label{abaCantor}
Let $C$ denote the standard middle-thirds Cantor set and $F_C$ the \emph{Cantor fan} $(C \times [0,1]) / (C \times \{1\})$.  It is known that a fan is smooth if and only if it is embeddable in $F_C$ (see \cite[Theorem 9, p.27]{Onfans}, \cite[Proposition 4, p.165]{Imacantorfan} and \cite[Corollary 4, p.90]{Eberhart}).

\label{F_CW}
We now construct the \emph{shrinking Cantor fan} $F_{C_\omega}$.  For each $i \in \N$ put
\[ F_i = \bigcup \left\{ tp_c: c \in C \cap \left[ \frac{(3)^{i-1}-1}{(3)^{i-1}}, \frac{(3)^i - 2}{(3)^i} \right] \textrm{ and } p_c = \left( c, \frac{1}{2^{i-1}} \right) \right\} .\]
Then $\{F_i\}_i$ is a sequence of fans homeomorphic to the Cantor fan tending to $\{t\}$ such that the intersection of any two of them is $\{t\}$.  Define
\[ F_{C_\omega} = \bigcup_{i \in \N} F_i .\]
Note that $F_{C_\omega}$ is a smooth fan, and $\cl{F_{C{\omega}}}{E(F_{C_\omega})} = E(F_{C_\omega}) \cup \{t\}$.

The following theorem gives us a characterization of $F_C$ and $F_{C_\omega}$.  We use this characterization to classify $\frac{1}{3}$-homogeneous smooth fans in Theorem \ref{aba13} and to prove that there is no $\frac{1}{4}$-homogeneous smooth fan in Theorem \ref{nosmooth1/4}.

\begin{thm}
\label{Cantor-ShrinkCantor}
Let $X$ be a smooth fan not locally connected such that $E(X)$ is an orbit of $X$.  Then:
\begin{enumerate}
  \item $X$ is homeomorphic to $F_C$ if and only if $E(X)$ is closed in $X$; and
  \item $X$ is homeomorphic to $F_{C_\omega}$ if and only if $\C \big( E(X) \big) = E(X) \cup \{t\}$.
\end{enumerate}
\end{thm}

\begin{proof}
In both 1) and 2) the left to right implication follows at once.  Now, for 1) assume that $E(X)$ is closed in $X$.  By Lemma \ref{lemanlc cantor}, $E(X)$ is a Cantor set.  We can suppose that $X \subset F_C$ and $E(X) \subset E(F_C) = C \times \{0\}$ (see \cite[Theorem 1 (2), p.74]{FansKelley}).  Now it is easy to see that $X$ is homeomorphic to $F_C$.

For 2) suppose that $\C \big( E(X) \big) = E(X) \cup \{t\}$.  As a consequence of Lemma \ref{lemanlc cantor}, $\C \big( E(X) \big)$ is a Cantor set.  So there is a sequence $\{E_i\}_i$ of disjoint, open and closed sets of $\C \big( E(X) \big)$ such that
\[ E(X) = \C \big( E(X) \big) - \{t\} = \bigcup_{i \in \N} E_i \hspace{0.3cm} \textrm{and} \hspace{0.3cm} \lim_{i \to \infty} E_i = \{t\} .\]

For each $i \in \N$, define $X_i = \bigcup_ {e \in E_i} te$.  Observe that for each $i \in \N$, the set $E(X_i) = E_i$ is a Cantor Set.  Hence $X_i$ is homeomorphic to the Cantor fan.

Given that $\{E_i\}_i$ converges to $\{t\}$ and $X$ is smooth, the sequence of fans $\{X_i\}_i$ converges to $\{t\}$.  Note that $X = \bigcup X_i$.  It is now straightforward to see that $X$ is homeomorphic to $F_{C_\omega}$.
\end{proof}

\section{$\frac{1}{3}$-homogeneous smooth fans}

Given $n \in \N - \{1,2\}$, the \emph{simple $n$-od} is the unique (up to homeomorphism) fan with exactly $n$ endpoints.  Simple $n$-ods are sometimes called \emph{finite fans}.  Clearly any simple $n$-od is $\frac{1}{3}$-homogeneous.

\label{fomegad}
We denote by $F_\omega$ the unique (up to homeomorphism) locally connected fan whose top has infinite order.  That is, $F_\omega$ is homeomorphic to $\bigcup_{i=1}^\infty ob_i$, constructed in $\m{R}^2$, where $o = (0,0)$, $b_i = \left( \frac{1}{i},\frac{1}{i^2} \right)$, and $ob_i$ is the straight line segment joining $o$ and $b_i$ for each $i \in \N$.  Clearly $F_\omega$ is $\frac{1}{3}$-homogeneous as well.

Because the Cantor set is homogeneous, it is straightforward to see that the Cantor fan $F_C$ and the shrinking Cantor fan $F_{C_\omega}$ are both $\frac{1}{3}$-homogeneous.

In 1961 A.\ Lelek constructed a fan as an example of a dendroid whose set of endpoints is dense and 1-dimensional (\cite[\textsection 9, p.314]{Lelek}).  Such a fan is called the \emph{Lelek fan}, and it is characterized as the only smooth fan whose set of endpoints is dense (see \cite[Corollary, p.33]{Lelekunique} and \cite[Theorem, p.529]{BulaOversteegen}).

In \cite{Aarts-Oversteegen93} Jan M.\ Aarts and Lex G.\ Oversteegen constructed the \emph{hairy arc} $H$.  This is a smooth dendroid constructed as the intersection of a sequence of subsets of $[0,1]^2$ containing the base $B = [0,1] \times \{0\}$ in such a way that the closure of each component of $H - B$ is an arc, called a \emph{hair}, these hairs are pairwise disjoint, the set of endpoints $E(H)$ is dense in $H$, and $E(H) \cup B$ is connected (\cite[Corollary 2.5, p.905]{Aarts-Oversteegen93}).  The following result is well known.

\begin{lem}
\label{H/B Lelek}
The Lelek fan is homeomorphic to $H/B$.
\end{lem}

\begin{proof}
Let $X = H/B$ and $q$ be the quotient function from $H$ to $X$.  Then $q$ identifies the base $B$ into a point $x_0 \in X$, and $q|_{H-B}$ is a homeomorphism from $H - B$ into $X - \{x_0\}$.  Hence $q$ is monotone and by \cite[Corollary 10, p.309]{SmoothDendroids}, $X$ is a smooth dendroid.  Since $R(H) \subset B$, we have $q(R(H)) = \{x_0\}$ and so $X$ is a smooth fan with top $x_0$.

We have $E(X)$ is dense in $X$, because $q|_{H-B}$ is a homeomorphism and $E(H)$ is dense in $H-B$.  Thus $X$ is homeomorphic to the Lelek fan.
\end{proof}

In \cite[Corollary 1.5, p.285]{Aarts-Oversteegen95} it is proved that there are exactly five homeomorphism types of points in $H$, namely
\begin{enumerate}
  \item endpoints of the base $B$,
  \item endpoints of hairs,
  \item hairless base points,
  \item base points with hair attached,
  \item interior points of hairs.
\end{enumerate}

It follows easily from this and Lemma \ref{H/B Lelek} that the Lelek fan is $\frac{1}{3}$-homogeneous.

We now prove that the fans discussed above comprise the complete list of $\frac{1}{3}$-homogeneous smooth fans.

\begin{thm}
\label{aba13}
A smooth fan is $\frac{1}{3}$-homogeneous if and only if it is homeomorphic to one of the following fans.
\begin{enumerate}
  \item A simple $n$-od, for some $n \in \N-\{1,2\}$,
  \item the dendrite $F_\omega$,
  \item the Cantor fan,
  \item the Lelek fan,
  \item the shrinking Cantor fan $F_{C_\omega}$.
\end{enumerate}
\end{thm}

\begin{proof}
It has been observed above that each of the fans in this list is $\frac{1}{3}$-homogeneous.

Suppose that $X$ is a $\frac{1}{3}$-homogeneous fan.  It is not difficult to see that the only locally connected $\frac{1}{3}$-homogeneous fans are $F_\omega$ and simple $n$-ods.

Assume, then, that $X$ is not locally connected.  If $E(X)$ is closed, from Theorem \ref{Cantor-ShrinkCantor} part 1), we have that $X$ is homeomorphic to the Cantor fan.  If $E(X)$ is dense, $X$ is homeomorphic to the Lelek fan by the characterization mentioned above.  Finally, suppose that $E(X)$ is neither closed nor dense in $X$.  We want to prove that $\C \big( E(X) \big) = E(X) \cup \{t\}$.  Suppose that $\C \big( E(X) \big) \cap O(X) \neq \emptyset$.  This implies $O(X) \subset \C \big( E(X) \big)$, because $E(X)$ and $O(X)$ are orbits of $X$.  It follows that $E(X)$ is dense, which contradicts our assumption.  Consequently, $\C \big( E(X) \big) \subset E(X) \cup \{t\}$ and since $E(X)$ is not closed, the equality holds.  By Theorem \ref{Cantor-ShrinkCantor} part 2), $X$ is homeomorphic to $F_{C_\omega}$.
\end{proof}

\section{$\frac{1}{4}$-homogeneous smooth fans}

In this section we prove the following result.

\begin{thm}
\label{nosmooth1/4}
There is no $\frac{1}{4}$-homogeneous smooth fan.
\end{thm}

We first establish two auxiliary results, which apply to both smooth and non-smooth fans.

\begin{lem}
\label{aba1/4}
Let $X$ be a $\frac{1}{4}$-homogeneous fan with top $t$.  Then $O(X) = O_1 \cup O_2$, where $O_1$ and $O_2$ are two orbits of $X$ and for each $e \in E(X)$, both $t$ and $e$ belong to $\C \big( O_1 \cap (te) \big)$ and to $\C \big( O_2 \cap (te) \big)$.
\end{lem}

\begin{proof}
It is easy to see that $E(X)$ must be one of the orbits of $X$ and $O(X)$ is the union of two orbits of $X$, $O(X) = O_1 \cup O_2$.  Let $e \in E(X)$.  Since $E(X)$ is an orbit of $X$, the two orbits $O_1$ and $O_2$ of ordinary points must both intersect $(te)$.

Suppose for a contradiction that $\C \big( O_1 \cap (te) )$ does not contain $e$.  Then there is a point $x \in (te)$ such that $O_1 \cap (xe) = \emptyset$, but for any $x' \in (tx)$ we have $O_1 \cap (x'x] \neq \emptyset$.  Let $y \in (xe)$, which implies $y \in O_2$.

If $x \in O_2$, then we can choose a homeomorphism $h: X \to X$ such that $h(y) = x$.  But then $h(x) \in (tx)$, and $h \big( (xe) \big) = \big( h(x)e \big) \subset O_2$, contradicting the choice of $x$.  Thus $x \in O_1$.  By a similar argument, we can see that $(tx] \not \subset O_1$.  But then if $z \in O_2 \cap (tx)$, we can choose a homeomorphism $h: X \to X$ such that $h(z) = y$, and then we have $h(x) \in O_1 \cap (ye) \subset O_1 \cap (xe) = \emptyset$, again a contradiction.  Therefore $e \in \C \big( O_1 \cap (te) )$.

The proofs that $t \in \C \big( O_1 \cap (te) \big)$ and $t,e \in \C \big( O_2 \cap (te) \big)$ are similar.
\end{proof}

\begin{lem}
\label{T,O-T}
If $X$ is a $\frac{1}{4}$-homogeneous fan, then either $O(X) \cap \C \big( E(X) \big) = \emptyset$ or $O(X) \subset \C \big( E(X) \big)$.
\end{lem}

\begin{proof}
Suppose for a contradiction that $X$ is a $\frac{1}{4}$-homogeneous fan such that $O_1 = O(X) \cap \C \big( E(X) \big)$ and $O_2 = O(X) - \C \big( E(X) \big)$ are both nonempty.  This means that these sets $O_1$ and $O_2$ must be orbits of $X$.

Take a point $x \in O_1$ and a sequence $\{e_n\}_n$ of points in $E(X)$ converging to $x$.  Denote by $e$ the endpoint such that $x \in te$.  From Lemma \ref{aba1/4}, we can choose $y \in (tx) \cap O_2$. Note that $tx \subset \lim_{n \to \infty} te_n$.  Hence, for each $n$ we can choose $y_n \in te_n$ such that $y_n \to y$.  Since $O_2 = X - \big( \{t\} \cup \C \big( E(X) \big) \big)$ is open in $X$, we can assume that $y_n \in O_2$ for all $n$.

Put $\varepsilon = \frac{1}{2} d(x,e)$.  Since $O_2$ is an open orbit, by Effros' Theorem ($\dagger$) there exists $\delta > 0$ such that, if $v \in O_2$ with $d(x,v) < \delta$, then there is an $\varepsilon$-homeomorphism of $X$ onto itself sending $y$ to $v$.  Choose $n \in \N$ such that $d(y,y_n) < \delta$ and $d(x,e_n) < \varepsilon$.  Then there is an $\varepsilon$-homeomorphism $h: X \to X$ such that $h(y) = y_n$.  Observe that $h(e) = e_n$.  But $d(e,e_n) \geq d(e,x) - d(x,e_n) > 2\varepsilon - \varepsilon = \varepsilon$, which is a contradiction since $h$ is a $\varepsilon$-homeomorphism.
\end{proof}

\begin{proof}[Proof of Theorem \ref{nosmooth1/4}]
Suppose for a contradiction that $X$ is a $\frac{1}{4}$-homogeneous smooth fan with top $t$.  It is easy to see that $X$ cannot be locally connected, as the only locally connected fans are simple $n$-ods and $F_\omega$.  Since $E(X)$ is an orbit of $X$ and $F_C$ and $F_{C_\omega}$ are $\frac{1}{3}$-homogeneous, by Theorem \ref{Cantor-ShrinkCantor}, it follows that $\C \big( E(X) \big)$ is not contained in $E(X) \cup \{t\}$.  Note that $E(X)$ is not dense, because in that case $X$ would be the Lelek fan which is also $\frac{1}{3}$-homogeneous.  It follows that $O(X) \cap \C \big( E(X) \big) \neq \emptyset$ and $O(X) \not \subset \C \big( E(X) \big)$, which contradicts Lemma \ref{T,O-T}.
\end{proof}

\section{Non-smooth fans}

For the remainder of this section, let $X$ be a fan with top $t$.

Given $a \in X$, we say that $X$ is \emph{smooth at $t$ with respect to} $a$ if for every sequence $\{a_n\}_n$ converging to $a$, we have $ta_n \to ta$.  Define the set
\[ S(X) = \{a \in X: X \textrm{ is smooth at } t \textrm{ with respect to } a \} .\]
From the definition, it is clear that $X$ is smooth if and only if $S(X) = X$, and $S(X)$ is invariant under homeomorphisms.  In \cite[Corollary 10, p.124]{SmoothandKelley} it is proved that $S(X)$ is a dense $G_\delta$-subset of $X$.

\begin{lem}
\label{Etsmooth}
If for each $x \in X - \{t\}$ we have $(tx) \cap S(X) \neq \emptyset$, then $E(X) \subset S(X)$.
\end{lem}

\begin{proof}
Suppose contrary to the claim that there is a point $e \in E(X) - S(X)$.  Then there is a sequence $\{e_n\}_n$ of points converging to $e$ such that the arcs $te_n$ converge to a continuum $Y \neq te$.  Observe that $Y \supsetneq te$.  Take $a \in Y - te$ and $b \in (te)$.  From our hypothesis we can assume that $a$ and $b$ belong to $S(X)$.  Since $a,b \in Y$, for each $n$ there are $a_n,b_n \in te_n$ such that $a_n \to a$ and $b_n \to b$.  Note that $ta_n \to ta$ and $tb_n \to tb$.  Since $b \notin ta$, we can suppose that $b_n \notin ta_n$ for all $n$.  This means that $ta_n \subset tb_n$ for each $n$.  Hence $ta \subset tb$.  This contradicts the choice of $a$ and $b$.
\end{proof}

In particular, Lemma \ref{Etsmooth} implies that if $O(X) \subset S(X)$ then $E(X) \subset S(X)$ as well.  It is easy to prove in this case that $t \in S(X)$ too.  Hence, we obtain the following result:

\begin{coro}
\label{Xsmooth}
If $X$ is a fan with $O(X) \subset S(X)$, then $X$ is smooth.
\end{coro}

A space $X$ is \emph{colocally connected} at a point $p$ in $X$, provided that each neighborhood of $p$ contains a neighborhood $V$ of $p$ such that $X - V$ is connected.

\begin{thm}
\label{13-homo nonsmooth}
If $X$ is a $\frac{1}{3}$-homogeneous non-smooth fan, then:
\begin{enumerate}
  \item $X$ is colocally connected at each point of $E(X)$;
  \item $E(X)$ is dense in $X$; and
  \item $S(X) = E(X) \cup \{t\}$.
\end{enumerate}
\end{thm}

\begin{proof}
Let $X$ be a $\frac{1}{3}$-homogeneous non-smooth fan.  The three orbits of $X$ are $\{t\}$, $O(X)$, and $E(X)$.  According to \cite[Theorem 3.5, p.235 and Theorem 4.1, p.237]{DendroidsK-Minc}, the set of endpoints at which $X$ is colocally connected is non-empty.  Because $E(X)$ is an orbit, we obtain 1).

From Corollary \ref{Xsmooth} and the fact that $O(X)$ is an orbit, we have that $O(X) \cap S(X) = \emptyset$.  Thus $S(X) \subset E(X) \cup \{t\}$.  Since $S(X)$ is dense, it intersects $E(X)$.  Hence,
\[ E(X) \subset S(X) \subset E(X) \cup \{t\} .\]
It follows that $E(X)$ is also dense.

Suppose for a contradiction that $t \notin S(X)$.  Then there is a sequence $\{t_n\}_n$ converging to $t$ such that the arcs $\{tt_n\}_n$ do not converge to $\{t\}$.  In this case we may assume, by taking a subsequence, that there exists a point $b \in O(X)$ and a sequence $\{b_n\}_n$ such that $b_n \in (tt_n)$ and $b_n \to b$.  Since $O(X)$ is an orbit, it follows that the same is true for every ordinary point; that is, for every $c \in O(X)$, there exist sequences $\{s_n\}_n$ and $\{c_n\}_n$ such that $c_n \in (ts_n)$, $s_n \to t$, and $c_n \to c$.  Since $O(X)$ is dense in $X$, we have the same property for endpoints $c \in E(X)$.  It is straightforward to see that this contradicts the fact that $X$ is colocally connected at each of its endpoints.  Therefore $t \in S(X)$.  This completes the proof of 3).
\end{proof}

\begin{thm}
\label{14-homo nonsmooth}
If $X$ is a $\frac{1}{4}$-homogeneous non-smooth fan, then:
\begin{enumerate}
  \item $X$ is colocally connected at each point of $E(X)$;
  \item $E(X)$ is dense in $X$; and
  \item $S(X) \supset E(X) \cup \{t\}$.
\end{enumerate}
\end{thm}

\begin{proof}
\setcounter{case}{0}
Let $X$ be a $\frac{1}{4}$-homogeneous non-smooth fan.  Since $E(X)$ is an orbit of $X$, we obtain 1) in exactly the same way as in the proof of Theorem \ref{13-homo nonsmooth}.

To prove $t \in S(X)$, we can use a similar argument as in the proof Theorem \ref{13-homo nonsmooth}.  The only difference is that in the present situation $O(X)$ is not an orbit of $X$, so instead we must use the result of Lemma \ref{aba1/4}, that for each $e \in E(X)$ the two orbits of ordinary points in $(xe)$ both accumulate on $e$.

It remains to prove that $E(X)$ is dense and $E(X) \subset S(X)$.

If $S(X) \cap O(X) = \emptyset$, then as in the proof of Theorem \ref{13-homo nonsmooth}, since $S(X)$ is dense it follows that $E(X) \subset S(X)$, and hence $E(X)$ is dense as well.

Assume for the remainder of the proof that $S(X) \cap O(X) \neq \emptyset$.  In this case the four orbits of $X$ are:
\[ \{t\}, \; E(X), \; O_1 = O(X) \cap S(X), \textrm{ and } O_2 = O(X) - S(X) .\]

From Lemmas \ref{aba1/4} and \ref{Etsmooth}, we conclude that $E(X) \subset S(X)$.  To prove $E(X)$ is dense, according to Lemma \ref{T,O-T} it suffices to prove that $\C \big( E(X) \big) \cap O(X) \neq \emptyset$, for in this case $O(X) \subset \C \big( E(X) \big)$ and hence $\C \big( E(X) \big) = X$.

Suppose for a contradiction that $\C \big( E(X) \big) \cap O(X)$ is empty.  Since $S(X)$ is dense, it follows that $O(X) \subset \C(O_1)$.  Hence, $O_1$ is dense in $X$.

We make the following claim:

\begin{itemize}
\item[($\star$)] For any $x \in O_2$, there exist $a,b \in O_1$ such that $x \in (ab) \subset O_2$.
\end{itemize}

To see this, suppose $x \in O_2$ and let $e \in E(X)$ be such that $x \in (te)$.  Since $x \notin S(X)$, there is a sequence of arcs $\{tx_n\}_n$ such that $x_n \to x$ but $tx_n \to Y \neq tx$, where $Y$ is a subcontinuum of $X$.  Since $t \in S(X)$, we must have $Y = tx'$ for some $x' \in (te)$ with $x \in (tx')$.  Clearly this implies that $[xx') \subset O_2$.  Since $O_2$ is an orbit, it follows that every point of $O_2$ is part of an interval in $O_2$.  The claim ($\star$) then follows from Lemma \ref{aba1/4}.

Now take any $x \in O_2$.  Because $O_1$ is dense, there is a sequence $\{a_n\}_n$ of points in $O_1$ such that $a_n \to x$.  For each $n$ let $e_n \in E(X)$ such that $a_n \in (te_n)$.  We can assume that $\{e_n\}_n$ is convergent to a point $e \in \C \big( E(X) \big) \subset E(X) \cup \{t\}$.  It follows easily from the fact that $E(X) \cup \{t\} \subset S(X)$ that $e \in E(X)$, and $x \in te = \lim_{n \to \infty} te_n$.

Recall that by \cite[Theorem 7.5, p.311]{Lelek}, $E(X)$ is a $G_\delta$ orbit, so we can apply Effros' theorem ($\dagger$) to it.  Put $\varepsilon = \frac{1}{2} d \big( x, O_1 \cap te) \big)$.  Observe that $\varepsilon > 0$ by ($\star$).  Let $\delta$ be given by Effros' theorem ($\dagger$) and choose $n \in \N$ such that $d(x,a_n) < \varepsilon$ and $d(e,e_n) < \delta$.  Then there is an $\varepsilon$-homeomorphism $h: X \to X$ such that $h(e_n) = e$.  Note that $h(te_n) = te$ and $d(x,h(a_n)) \geq 2\varepsilon$, because $h(a_n) \in O_1 \cap te$.  So
\[ d \big( a_n,h(a_n) \big) \geq d \big( x,h(a_n) \big) - d(x,a_n) \geq 2\varepsilon - \varepsilon = \varepsilon .\]
This is a contradiction since $d \big( p,h(p) \big) < \varepsilon$ for all $p \in X$.  Hence, $O(X) \cap \C \big( E(X) \big) \neq \emptyset$.

Now, from Lemma \ref{T,O-T} we have that $O(X) \subset \C \big( E(X) \big)$.  Therefore $E(X)$ is dense.
\end{proof}

\section{Questions}

Theorems \ref{13-homo nonsmooth} and \ref{14-homo nonsmooth} provide some information about how a $\frac{1}{3}$-homogeneous or $\frac{1}{4}$-homogeneous non-smooth fan would have to look, but as yet we are not aware of any examples of such objects.  To complete the classification of all $\frac{1}{3}$-homogeneous or $\frac{1}{4}$-homogeneous fans, one must answer the following question:

\begin{prob}
\label{nonsmooth exist}
Does there exist a $\frac{1}{3}$-homogeneous non-smooth fan?  a $\frac{1}{4}$-homogeneous non-smooth fan?
\end{prob}

We have confined our attention here to fans, as they comprise a simple class of dendroids, and hence make a good starting point for the study of homogeneity degree of dendroids.  In the classification of all $\frac{1}{3}$-homogeneous dendrites given in \cite{13-dendrites}, many interesting dendrites appear, including the locally connected fans, the universal ($n$-branching) dendrites, and the ($n$-branching) Gehman dendrites.  It would be interesting to extend this to a classification of all $\frac{1}{3}$-homogeneous dendroids, perhaps beginning with smooth dendroids.

\begin{prob}
\label{13-homo dendroids}
What are all the $\frac{1}{3}$-homogeneous (smooth) dendroids?
\end{prob}

\bibliographystyle{amsplain}
\bibliography{FansHomogeneity}

\end{document}